\documentclass[11pt, a4paper]{article}

\usepackage{epsfig}
\usepackage{amssymb}
\usepackage{amsthm}
\usepackage{amsmath}
\usepackage{float}
\usepackage{multicol}
\usepackage{multirow}
\usepackage{cite}
\usepackage{enumerate}
\usepackage{xcolor}
\usepackage{tabu}

\def\ms{\medskip}
\def\nt{\noindent}

\definecolor{vividviolet}{rgb}{0.62, 0.0, 1.0}

\def\Z{\mathbb Z}
\def\di{\displaystyle}
\newtheoremstyle{de}
  {10pt}          
  {10pt}  
  {\rm}  
  {}
  {\bf}  
  {. }    
  { }    
  {}     
\theoremstyle{de}

\newtheorem{problem}{Problem}[section]

\newtheoremstyle{theorem}
  {10pt}          
  {10pt}  
  {\it}  
  {}
  {\bf}  
  {. }    
  { }    
  {}     
\theoremstyle{theorem}

\newtheorem{theorem}{Theorem}[section]
\newtheorem{lemma}[theorem]{Lemma}

\newtheorem{corollary}[theorem]{Corollary}

\newtheorem{conjecture}{Conjecture}[section]

\numberwithin{equation}{section}
\def\Z{\mathbb{Z}}

\def\a{\alpha}
\def\b{\beta}
\def\s{\sigma}
\def\r{\rho}

\def\n{\nu}

\begin{document}
\begin{center}
{\mathversion{bold}\Large \bf A novel approach to determining chromatic number induced by labelings}

\bigskip
{\large  Gee-Choon Lau$^{a}$, Wai Chee Shiu$^b$,  Zhen Bin Gao$^{c,}$}\footnote{the corresponding author}

\medskip

\emph{{$^a$}77D, Jalan Suboh, 85000 Segamat, Johor, Malaysia}\\
\emph{geeclau@yahoo.com}\\

\medskip
\emph{{$^b$}Department of Mathematics,}\\
\emph{The Chinese University of Hong Kong,}\\
\emph{Shatin, Hong Kong, P.R. China.}\\
\emph{wcshiu@associate.hkbu.edu.hk}\\

\medskip
\emph{{$^c$}College of General Education,}\\
\emph{Guangdong University of Science and Technology,}\\
\emph{Dongguan, 523083, P.R. China.}\\
\emph{gaozhenbin@aliyun.com}\\



\end{center}
\begin{abstract}
Given a simple graph $G=(V,E)$ of order $p$ and size $q$, a bijection $f : V\cup E \to \{1, 2, \ldots, p+q\}$ is a local total neighborhood antimagic labeling of $G$ if the induced vertex coloring has the property $f^+_{tn}(u) \ne f^+_{tn}(v)$ for every two adjacent vertices $u$ and $v$ where $f^+_{tn}(u) = \sum (f(ux) + f(x))$  over every neighbor $x$ of $u$. The local total neighborhood antimagic chromatic number of $G$, denoted $\chi_{ltna}(G)$ is the minimum number of distinct induced colors over all local total neighborhood antimagic labeling of $G$. In this paper, we determine the local total neighborhood antimagic chromatic number of the join of graphs with distinct parity orders.

\ms


 \end{abstract}

\section{Introduction}

Let $G=(V,E)$ be a simple $(p,q)$-graph of order $p$ and size $q$. A bijection $f : E \to \{1,2,\ldots, q\}$ is an {\it antimagic} labeling of $G$ if all the vertices have distinct induced vertex weights where the weight of a vertex $u$ is given by $w(u) = \sum f(e)$ with $e$ ranging over all the edges incident to $u$ \cite{H+R}.  The famous unsolved problems are as follows.

\begin{conjecture} All connected graphs except $K_2$ are antimagic. \end{conjecture}

\begin{conjecture} All trees except $K_2$ are antimagic. \end{conjecture}

\nt In~\cite{Arumugam}, the authors introduced the concept of {\it local antimagic} labeling that allows non-adjacent vertices to have the same induced vertex weights. Thus, $G$ is local antimagic if there is a bijective labeling $f: E \to \{1,2,\ldots,q\}$ such that $w(u) \ne w(v)$ for every edge $uv$ of $G$. Clearly, this labeling induces a proper vertex coloring. Therefore,  Arumugam et al. called the minimum number of distinct vertex colors over all local antimagic labelings of $G$ as the local antimagic chromatic number of $G$, denoted $\chi_{la}(G)$. Interested readers may refer to~\cite{Lau+Ng+Shiu, Lau+Shiu+Nal+Zhang+Prem, Lau+Shiu+Ng-IJMSI} for many related results. In particular, one may refer to~\cite{Lau+Prem+Shiu+Nal, Lau+Shiu+Ng-DMGT, Lau+Shiu-AMH, Lau+Shiu-AC, Lau+Shiu-CN, Lau+Shiu-UM} for the local antimagic chromatic number of the join of graphs.

\nt In~\cite{S+P, S+P+P}, the study of total neighborhood antimagic labeling was initiated. As a natural variation, a bijective total labeling $f : V \cup E \to \{1,2,\ldots,p+q\}$ is called a {\it local total neighborhood antimagic labeling} of $G$ if $f$ induces a vertex coloring $f^+_{tn} : V \to \Z$, where $f^+_{tn}(u) = \sum\limits_{ux\in E} [f(ux) + f(x)]$. If $u$ is an isolated vertex, we let $f^+_{tn}(u) = 0$.
We say $G$ is {\it local total neighborhood antimagic} if it admits a {\it local total neighborhood antimagic labeling} $f$~\cite{GLSY}.

\ms\nt We say $f^+_{tn}(u)$ is the induced color of $u$ under $f$. For a local total neighborhood antimagic labeling $f$, the number of distinct induced colors is denoted $c_{tn}(f)$  so that $f$ is also a local total neighborhood antimagic $c_{tn}(f)$-coloring of $G$.  The {\it local total neighborhood antimagic chromatic number} of $G$, denoted $\chi_{ltna}(G)$, is
\[\min\{c_{tn}(f) \;|\; f \mbox{ is a local total neighborhood antimagic labeling of } G\}.\]


\ms\nt Thus, throughout this paper, we only focus on the graph $G$ which is a simple graph of order $p\ge 2$ without isolated vertices and admits a local total neighborhood antimagic labeling $f$, unless stated otherwise.  Moreover,
\begin{equation}\label{eq-obvious}
\chi_{ltna}(G)\ge \chi(G) \ge 2.
\end{equation} Similar to the local antimagic conjecture, we also have

\begin{conjecture} Every graph admits a local total neighborhood antimagic labeling. \end{conjecture}

\nt For integers $a < b$, let $[a,b] = \{a, a+1, \dots, b\}$. The join of graphs $G$ and $H$ is denoted $G\vee H$ with $V(G\vee H) = V(G) \cup V(H)$ and $E(G\vee H) = E(G) \cup E(H) \cup \{uv \mid u\in V(G), v\in V(H)\}$. For a given vertex $x$ of $G$, let $d_G(x)$ (or $d(x)$ if no ambiguity) be the degree of $x$. Interested readers may refer to~\cite{GLSY} for some general bounds and exact local total neighborhood antimagic chromatic number of some standard graphs. In this paper, we give  structured matrix constructions that allows us the determine the local total neighborhood antimagic chromatic number of the join of graphs with distinct parity orders. 


\section{Join graphs with orders of distinct parity}

\nt Suppose $G$ is a $(p,q)$-graph for even $p\ge 2$. Consider the following conditions:
\begin{enumerate}[(a)]
\item For $t\ge 2$, $G$ is a $t$-partite graph with partite sets $V_l, 1\le l\le t$, such that $|V_l| = s_l$ with $s_{l+1} \ge s_l\ge 1$.
\item $G$ admits a bijective total labeling $g : V(G) \cup E(G) \to [1,p+q]$ with induced vertex color of $u$ given by $g^+_{tn}(u) = \sum\limits_{x\in N(u)} [g(ux) + g(x)]$ such that (i) if $s_l=1$, then the only weight in $V_l$ is $\a_l$, and (ii) if $s_l\ge 2$, then all the weights in $V_l$ form an arithmetic sequence with first term $\a_l$ and common difference $d_l$.
\end{enumerate}

\nt {\bf Condition~(A): } For even $p\ge 2$ and odd $m\ge 3$, there exists a $p\times m$ matrix $A_{p\times m}$ such that the set of its entries is $[1,pm]$ and every column sum is a constant $c = p(pm+1)/2$, and all the row sums can be partitioned into $t$ sets $S_l$, $1\le l\le t$, of size $s_l$ such that (i) if $s_l = 1$, then the only element in $S_l$ is $\b_l$, and (ii) if $s_l\ge 2$, then all the elements in $S_l$ form an arithmetic sequence with first term $\b_l$ and common difference $-d_l$.

\begin{theorem}\label{thm-GVOm-1} Suppose $G$ is a $(p,q)$-graph with $V(G) = \{u_i\;|\; 1\le i\le p\}$  satisfying conditions (a), (b) and (A) above. If for $1\le l<l'\le t$,
\begin{enumerate}[(i)]
\item$\a_l + \b_l \ne \a_{l'} + \b_{l'}$, and
\item $\a_l + \b_l + m(2p+2q+pm) + m(m+1)/2 \ne \sum\limits^p_{i=1} g(u_i) + p(pm+1)/2 + p(p+q)$,
\end{enumerate}
then $\chi_{ltna}(G \vee O_m) \le t+1$. The equality holds if $\chi(G)=t$.
\end{theorem}

\begin{proof} Let $G$ be as defined above. Suppose $F = G\vee O_m$. Let $V(O_m) = \{v_j\;|\; 1\le j\le m\}$. Define a bijective total labeling $f : V(F) \cup E(F) \to [1,p+q+pm+m]$ such that
\begin{enumerate}[(1)]
\item $f(x) = g(x)$ for $x\in V(G)\cup E(G)$,
\item $f(u_iv_j) = a_{i,j} + p+q$, where $a_{i,j}$ is the $(i,j)$-entry of $A_{p\times m}$, $1\le i\le p$, $1\le j\le m$,
\item $f(v_j) = p+q+pm + j$, $1\le j\le m$.
\end{enumerate}
For $l\ge 2$, let $\s_l = s_1+s_2+\cdots + s_{l-1}$ and $\s_0=0$. Without loss of generality, we let $g^+_{tn}(u_{\s_l+k}) = \a_l+(k-1)d_l$ for $1\le l\le t$, $1\le k\le s_l$, i.e., $u_{\s_l+k}$ is the $k$-th vertex in $V_l$ that corresponds to the $k$-th term of the corresponding arithmetic sequence under the labeling $g$ in (b). Similarly, for $1\le l\le t$, $1\le k\le s_l$, we let $\r_{\s_l+k}$ be the $(\s_l+k)$-th row sum of $A_{p\times m}$ that corresponds to the $k$-th term of the arithmetic sequence with first term $\b_l$ and common difference $-d_l$. Thus, $\r_{\s_l+k} = \b_l-(k-1)d_l$.
Since $f(x) = g(x)$ for $x\in V(G)$, we have
\begin{align*}
f^+_{tn}(u_{\s_l+k}) & = g^+_{tn}(u_{\s_l+k}) +  \r_{\s_l+k} + m(p+q) + \sum^m_{j=1} f(v_j)\\
       &= g^+_{tn}(u_{\s_l+k}) +  \r_{\s_l+k} + m(p+q) + m(p+q+pm) + m(m+1)/2\\
      & = \a_l + \b_l + m(2p+2q+pm) + m(m+1)/2, \quad 1\le k\le s_l.\\
f^+_{tn}(v_j) & = \sum^p_{i=1} g(u_i) + p(pm+1)/2 + p(p+q).
\end{align*}
\nt Therefore, all the vertices in $V_l$ have the same vertex color $\a_l + \b_l + m(2p+2q+pm) + m(m+1)/2$. By (i), we have $f^+_{tn}(u_{\s_l+k}) \ne f^+_{tn}(u_{\s_{l'}+k'})$ for all $1\le l< l'\le t$,  $1\le k\le s_l$, $1\le k' \le  s_{l'}$. By (ii), we have $f^+_{tn}(u_{\s_l+k}) \ne f^+_{tn}(v_j)$ for $1\le l\le t$ and $1\le j \le m$. Thus, $f$ is a local total neighborhood antimagic $(t+1)$-coloring of $F$ so that $\chi_{ltna}(F) \le t+1$. Since $\chi_{ltna}(F) \ge \chi(F) = \chi(G)+1$, the equality holds if $\chi(G) = t$.
\end{proof}

\nt We shall now define a suitable $p\times m$ array to determine the exact $\chi_{ltna}(G\vee O_m)$ for some standard graphs $G$. For $p = 2r$, $m = 2s+1$, $r,s\ge 1$, let $M_k$, $1\le k\le s$, be the following $2r \times 2$ array with rows $R_1$ to $R_{2r}$, and columns $C_{2k-1}, C_{2k}$.
\[\fontsize{9}{13}\selectfont
\begin{tabu}{|c|[1pt]c|c|}\hline
      &   C_{2k-1}   &  C_{2k}   \\\tabucline[1pt]{-}
R_1 & 2r(k-1)+1 & 2r(2s-k+1)+r   \\\hline
R_2 & 2r(k-1)+2 & 2r(2s-k+1)+r-1   \\\hline
\vdots & \vdots & \vdots  \\\hline
R_r & 2r(k-1)+r  &  2r(2s-k+1)+1  \\\tabucline[1pt]{-}
R_{r+1}  & 2r(2s-k+2)   & 2r(k-1)+r+1\\\hline
R_{r+2} & 2r(2s-k+2) -1     &  2r(k-1)+r+2 \\\hline
\vdots & \vdots & \vdots  \\\hline
R_{2r}  &  2r(2s-k+2)-r+1    &  2r(k-1)+2r   \\\hline
\end{tabu}\]

\nt Let $M = \begin{pmatrix}  M_1  & M_2 & \cdots &M_s \end{pmatrix}$ with $(i,j)$-entry $m_{i,j}$. We now have the following observations.
\begin{enumerate}[({M}1)]
\item Set of entries in $M_k$ is $[2r(k-1)+1, 2rk]\cup [2r(2s-k+1)+1, 2r(2s-k+2)]$. Thus, set of entries in $M$ is $[1, 2rs] \cup [2r(s+1)+1, 2r(2s+1)]$.
\item Each column of $M$ has constant column sum $r[2r(2s+1)+1] = r(4rs+2r+1)$.
\item For $1\le i\le r$, $m_{i,2k-1}+m_{i,2k} = 2r(2s)+r+1$ and $m_{r+i,2k-1} + m_{r+i,2k} = 2r(2s+1)+r+1$. Thus, the $i$-th (and $(r+i)$-th) row sum of $M$ is a constant $s[2r(2s)+r+1] = s(4rs+r+1)$ (and $s[2r(2s+1)+r+1] = s(4rs+3r+1)$).
\end{enumerate}

\nt For $1\le i\le r$, suppose $A_1$ is a $2r\times 1$ array with the $i$-th entry $2rs+2i-1$, and the $(r+i)$-th entry $2rs+2i$. Clearly, $A_1$ has column sum $r(4rs+2r+1)$ as $M$. Let $A = \begin{pmatrix} M & A_1\end{pmatrix}$ be a $2r \times (2s+1)$ array. For $1\le k\le s$, $1\le i\le r$, the following properties hold.
\begin{enumerate}[({A}1)]
\item There is a one-to-one correspondence between entries of $A$ and $[1,2r(2s+1)]$.
\item $A$ has constant column sum $r(4rs+2r+1)$.
\item From (M3), the $i$-th row of $A$ has row sum $s(4rs+r+1) + 2rs+2i-1 = s(4rs+3r+1)+2i-1$.  Therefore, these sums form an arithmetic sequence with first term $\b_1 = s(4rs+3r+1)+1$ and common difference $-d_1=2$.
\item From (M3), the $(r+i)$-th row of $A$ has row sum $ s(4rs+3r+1) + 2rs+2i = s(4rs+5r+1)+ 2i$. Therefore, these sums form an arithmetic sequence with first term $\b_2 =s(4rs+5r+1)+ 2$ and common difference $-d_2=2$.
\end{enumerate}

\nt For $1\le i\le 2r$, suppose $B_1$ is a $2r \times 1$ array with the $i$-th entry $2rs+i$. Clearly, $B_1$ also has column sum $r(4rs+2r+1)$ as $A_1$ and $M$. Let $B = \begin{pmatrix} M & B_1\end{pmatrix}$ be a $2r \times (2s+1)$ array. For $1\le k\le s$, $1\le i\le r$, we have the following observations.
\begin{enumerate}[({B}1)]
\item There is a one-to-one correspondence between entries of $B$ and $[1,2r(2s+1)]$.
\item $B$ has constant column sum $r(4rs+2r+1)$.
\item From (M3), the $i$-th row of $B$ has row sum $s(4rs+r+1) + 2rs+i = s(4rs+3r+1)+i$. Therefore, these sums form an arithmetic sequence with first term $\b_1 = s(4rs+3r+1)+1$ and common difference $-d_1=1$.
\item From (M3), the $(r+i)$-th row of $B$ has row sum $s(4rs+3r+1) + 2rs+ r+i = s(4rs+5r+1)+r+i$. Therefore, these sums form an arithmetic sequence with first term $\b_2 = s(4rs+5r+1)+r+1$ and common difference $-d_2=1$.
\end{enumerate}
\nt Note that when $r=1$, we get \[A=B=\begin{pmatrix}1 & 4s+1 & 3 & 4s-1  & \cdots & 2s-1 & 2s+3 & 2s+1\\
4s+2 & 2 & 4s & 4 & \cdots & 2s+4 & 2s & 2s+2
\end{pmatrix}\]
The first and second row sums are $s(4s+2)+2s+1 = (2s+1)^2$ and $s(4s+4)+2s+2 =(2s+1)(2s+2)$, respectively. Every column sum is a constant $4s+3$. All these sums are consistent with the sums obtained for $r>1$.

\ms\nt In~\cite[Theorem 4.13]{GLSY}, the authors prove that $\chi_{ltna}(rK_2\vee K_1) = 3$ for $r\ge 1$. We now extend this to $rK_2\vee O_{2s+1}$ for $r, s\ge 1$.

\begin{theorem}\label{thm-rK2VOm} For $r, s\ge 1$, $\chi_{ltna}(rK_2 \vee O_{2s+1}) = 3$.   \end{theorem}

\begin{proof} Let $G=rK_2$ with vertex set $\{u_i\mid 1\le i\le2 r\}$ and edge set $\{u_iu_{r+i}\mid 1\le i\le r\}$. Thus, $G$ is a bipartite graph with partite sets $V_1 = \{u_i \mid 1\le i\le r\}$, $V_2=\{u_{r+i} \mid 1\le i\le r\}$ and size $s_1 = s_2 = r$ satisfying condition (a) of Theorem~\ref{thm-GVOm-1}. Define a total labeling $g : V(G)\cup E(G) \to [1, 3r]$ such that $g(u_i) = 3r+1-i$, $g(u_{r+i}) = 2r+1-i$, and $g(u_iu_{r+i}) = r+1-i$, $1\le i\le r$. Thus, for $r\ge 2$, the induced vertex color of vertices in $V_1$ (respectively, in $V_2$) form an arithmetic sequence with first term $\a_1=3r$ and common difference $d_1=-2$ (respectively, $\a_2=4r$ and common difference $d_2=-2$). If $r=1$, $V_1$ (and $V_2$) has only vertex $u_1$ (and $u_2$) with induced vertex color $\a_1=3$ (and $\a_2=4$). Therefore, $G$ also satisfies condition (b) of Theorem~\ref{thm-GVOm-1}.

\ms\nt Let $p=2r$ and $m=2s+1$, $r,s\ge 1$. We have defined the matrix $A$ above which satisfies Condition~(A). Note that, when $r\ge 2$, $\b_1 =  s(4rs+3r+1)+1$ and  $-d_1 = 2$; $\b_2 =  s(4rs+5r+1)+2$ and $-d_2=2$. When $r=1$ we also have the same $\b_1$ and $\b_2$.

\ms\nt Now we are going to show conditions (i) and (ii) of Theorem~\ref{thm-GVOm-1} hold. We first have
$$\a_1 + \b_1 = 3r +  s(4rs+3r+1)+1 \ne 4r +  s(4rs+5r+1)+2 = \a_2 + \b_2.$$ Thus, condition (i) holds.

\nt We let $L_l$, $l=1,2$, be the left hand side and $R$ be the right hand side of the inequality of condition (ii). Thus,
\begin{align*} L_1  &=  \a_1 + \b_1 + m(2p+2q+pm) + m(m+1)/2 \\ &= 3r +  s(4rs+3r+1)+1 +(2s+1)(6r + 2r(2s+1)) + (2s+1)(s+1) \\
&= r(12s^2 + 23s + 11) + 2(s + 1)^2,\\
L_2 & =  \a_2+\b_2 + m(2p+2q+pm) + m(m+1)/2 \\ &= 4r +  s(4rs+5r+1)+2 + (2s+1)(6r + 2r(2s+1)) + (2s+1)(s+1) \\ &= r(12s^2 + 25s + 12) + 2(s + 1)^2 + 1,\\
R & =  \sum\limits^p_{i=1}g(u_i) + p(pm+1)/2 + p(p+q) \\
&= r(4r+1) + r(4rs+2r+1) + 2r(3r) \\ &= 4r^2s+12r^2+2r.
\end{align*}
Now,
\begin{align*}
L_1-R& =12rs^2+23rs+9r+2s^2+4s-4r^2s-12r^2+2\\
& = 4rs(3s-r)+(2s-r)(s+12r-9)+22s+2.
\end{align*}
Clearly, when $2s\ge r$, $\di L_1-R>0$.

\nt When $3s-4\ge r>2s$, then $3s-r\ge 4$ and $2s-r\ge -s+4$. We have
\begin{align*}L_1-R & \ge 16rs+(-s+4)(s+12r-9)+22s+2= 4rs-s^2+48r+35s-34\\
& > 8s^2-s^2+48r+35s-34>0.
\end{align*}
\nt When $r=3s-3$, $L_1-R=-s^2+142s-133$.

\nt When $r=3s-2$, $L_1-R=-13s^2+113s-64$.

\nt When $r=3s-1$, $L_1-R=-25s^2+76s-19$.

\nt One may check that the determinants of the above quadratic forms are not perfect square. Thus the above 3 differences are not zero.

\nt When $3s\le r$, then $2s-r\le -s$ and $-r\le-3s$. Thus
\begin{align*}L_1-R & \le -s(s+12r-9)+22s+2=-12rs-s^2+31s+2\\
&\le -36s^2-s^2+31s+2=-37s^2+31s+2<0.
\end{align*}
Next consider
\begin{align*}
L_2-R & = 12rs^2+25rs+10r+2s^2+4s-4r^2s-12r^2+3\\
& = 4rs(3s-r)+2rs+(2s-r)(s+12r-10)+24s+3
\end{align*}
Since $L_2>L_1$, $L_2 - R > 0$ when $r\le 3s-4$.

\nt When $r=3s-3$, $L_2-R=5s^2+139s-135>0$.

\nt When $r=3s-2$, $L_2-R=-7s^2+112s-65$.

\nt When $r=3s-1$, $L_2-R=-19s^2+77s-19$.

\nt When $r=3s$, $L_2-R=-31s^2+34s+3$

\nt One may check that the determinants of the above quadratic forms are not perfect square. Thus the above 3 differences are not zero.

\nt When $3s+1\le r$. This implies $3s-r\le -1$, $2s-r< -s$ and $-r<-3s$. Thus
\begin{align*}L_2-R & < -2rs-s(s+12r-10)+24s+3=-s^2+34s-14rs+3\\
&\le -s^2+34s-42s^2+3=-43s^2+34s+3<0.
\end{align*}
By Theorem~\ref{thm-GVOm-1} and $\chi(rK_2\vee O_{2s+1})=3$, we have $\chi_{ltna}(rK_2\vee O_{2s+1})=3$.\end{proof}

\ms\nt In~\cite[Theorem 4.19]{GLSY}, the authors prove that $\chi_{ltna}(C_{4n+2} \vee K_1) = 3$ for $n\ge 1$. We now extend this to $C_{4n+2} \vee O_{2s+1}, s\ge 1$.

\begin{theorem}\label{thm-C4n+2VO2s+1} For $n, s\ge 1$, $\chi_{ltna}(C_{4n+2}\vee O_{2s+1}) = 3$. \end{theorem}

\begin{proof} Let $V(C_{4n+2}) = \{u_i \mid 1\le i\le 4n+2\}$ and $E(C_{4n+2}) = \{u_iu_{i+1} \mid 1\le i\le 4n+2\}$ with $u_{4n+3} = u_1$. Thus, $C_{4n+2}$ is a bipartite graph with partite sets $V_1 = \{u_{2i} \mid 1\le i\le 2n+1\}$ and $V_2 = \{u_{2i-1} \mid 1\le i\le 2n+1\}$ and size $s_1 = s_2 = 2n+1$ satisfying condition (a) of Theorem~\ref{thm-GVOm-1}. Define a bijective total labeling $g : V(C_{4n+2}) \cup E(C_{4n+2}) \to [1, 8n+4]$ as in the proof of Theorem~4.19 in~\cite{GLSY}. For completeness, $g$ is given below.
\begin{multicols}{2}
\begin{enumerate}[(1)]
\item $g(u_{4i-3}) = i$ for $1\le i\le n+1$,
\item $g(u_{4i-1}) = n+1+i$ for $1\le i\le n$,
\item $g(u_{4i-2}) = 3n+3-i$ for $1\le i\le n+1$,
\item $g(u_{4i}) = 4n+3-i$ for $1\le i\le n$,
\item $g(u_{2i+1}u_{2i+2}) = 4n+2+i$ for $1\le i\le 2n$,
\item $g(u_1u_2) = 6n+3$, $f(u_2u_3) = 6n+4$,
\item $g(u_{2i+2}u_{2i+3}) = 8n+5-i$ for $1\le i\le 2n$.
\end{enumerate}
\end{multicols}

\nt Now, consider the induced vertex color of vertices in $V_1$.
\begin{align*}
g^+_{tn}(u_2) & = g(u_3) + g(u_2u_3) + g(u_1) + g(u_1u_2) = (n+2) + (6n+4) + 1 + (6n+3) \\&= 13n+10,\\
g^+_{tn}(u_{4n+2}) & = g(u_1) + g(u_{4n+2}u_1) + g(u_{4n+1}) + g(u_{4n+2}) \\&= 1 + (6n+5) + (n+1) + (6n+2) = 13n+9.\\  \mbox{Similarly, we have \quad }&\\
g^+_{tn}(u_{4i})& = 13n+9+2i \mbox{ for }1\le i\le n,\\
g^+_{tn}(u_{4i+2}) & = 13n+10+2i\mbox{ for }1\le i\le n-1.
\end{align*}
Thus $\{g^+_{tn}(u_{4n}), g^+_{tn}(u_{4n-2}),\dots, g^+_{tn}(u_{4}), g^+_{tn}(u_2), g^+_{tn}(u_{4n+2})\}$ forms an arithmetic sequence with first term $\a_1=15n+9$ and $d_1=-1$ when $n\ge 2$. When $n=1$, the sequence is\\ $\{g^+_{tn}(u_{4}), g^+_{tn}(u_{2}), g^+_{tn}(u_{6})\}=\{24,23,22\}$ with $\a_1=24$ and $d_1=-1$.

\ms \nt Next, consider the induced vertex color of vertices in $V_2$. Similarly we have
\begin{multicols}{2}
$\begin{aligned}
g^+_{tn}(u_1)& = 17n+12,\\
g^+_{tn}(u_3) &  = 17n+11,\\
\end{aligned}$

$\begin{aligned}
g^+_{tn}(u_{4i-3}) & = 19n+15-2i \mbox{ for } 2\le i\le n,\\
g^+_{tn}(u_{4i-1}) & = 19n+14-2i \mbox{ for } 2\le i\le n,\\
g^+_{tn}(u_{4n+1}) & = 17n+13.
\end{aligned}$
\end{multicols}
\nt Thus $\{g^+_{tn}(u_{5}), g^+_{tn}(u_{7}),\dots, g^+_{tn}(u_{4n-1}), g^+_{tn}(u_{4n+1}), g^+_{tn}(u_{1}), g^+_{tn}(u_3)\}$ forms an arithmetic sequence with first term $\a_2=19n+13$ and $d_2=-1$ when $n\ge 2$. When $n=1$, the sequence is $\{g^+_{tn}(u_{5}), g^+_{tn}(u_{1}), g^+_{tn}(u_{3})\}=\{30, 29, 28\}$ with $\a_2=30$ and $d_2=-1$.

\nt Thus, $C_{4n+2}$ also satisfies condition (b) of Theorem~\ref{thm-GVOm-1}.

\ms\nt Let $p=2r=4n+2$ and $m=2s+1$, $n,s\ge 1$. We have defined the matrix $B$ above which satisfies Condition (A) with $\b_1 = s[4s(2n+1)+3(2n+1)+1]+1 = s(8ns+4s+6n+4)+1$, $\b_2 =  s[4s(2n+1)+5(2n+1)+1]+r+1 = s(8ns+4s+10n+6)+r+1$ and common difference $-d_1 = -d_2 = 1$.

\ms\nt We are now going to show conditions (i) and (ii) of Theorem~\ref{thm-GVOm-1} hold. We also let $L_l$, $l = 1,2$ be the left hand side and $R$ be the right hand side of the inequality of condition (ii).

\nt If $n=1$, then $r=3$ and $p=q=6$. We first have $$\a_1+\b_1 = 24 +  s(12s+10)+1 \ne 30+ s(12s+16)+4=\a_2+\b_2.$$ Thus, condition (i) holds. We next have
\begin{align*}
L_1 & =  \a_1 + \b_1 + m(2p+2q+pm) + m(m+1)/2 \\ &=12s^2+10s+25 + (2s+1)(24+6(2s+1))  + (2s+1)(s+1) \\
 & =38s^2+85s+56,\\
L_2 & = \a_2+\b_2+ m(2p+2q+pm) + m(m+1)/2 \\ & = 12s^2+16s+34 + (2s+1)(24+6(2s+1))  + (2s+1)(s+1) \\
& = 38s^2+91s+65,\\
R & = \sum\limits^p_{i=1}g(u_i) + p(pm+1)/2 + p(p+q) \\ & = 21 + 3(3(2s+1)+1)+3(6) = 18s+51.
\end{align*}
Clearly, condition (ii) holds.

\ms\nt We now consider $n\ge 2$. First we have $$\a_1+\b_1 = (15n+9)+s(8ns+4s+6n+4)+1 \ne (19n+13)+s(8ns+4s+10n+6)+(2n+1)+1 = \a_2+\b_2.$$ Thus, condition (i) holds. We next have
\begin{align*}
L_1  & =  \a_1 + \b_1 + m(2p+2q+pm) + m(m+1)/2 \\ &= 15n+10+s(8ns+4s+6n+4)+(2s+1)(2(8n+4) + (4n+2)(2s+1))+(2s+1)(s+2)\\
&=(24s^2+54s+35)n+14s^2+33s+22, \\
L_2 & = \a_2 + \b_2 + m(2p+2q+pm) + m(m+1)/2 \\ &= 21n+15+ s(8ns+4s+10n+6) + (2s+1)(2(8n+4) + (4n+2)(2s+1)) +(2s+1)(s+2) \\
& = (24s^2+58s+41)n+14s^2+35s+27,\\
R & = \sum\limits^p_{i=1}g(u_i) + p(pm+1)/2 + p(p+q) \\ & = (2n+1)(4n+3) + (2n+1)((4n+2)(2s+1)+1)+(4n+2)(8n+4)\\
 & = (16s+48)n^2 + (16s+52)n+4s+14.
\end{align*}
\begin{align*}
L_1-R& =24ns^2+38ns+14s^2+29s-16n^2s-48n^2-17n+8\\
& = 8ns(3s-2n)+(14s+48n+17)(s-n)+4ns+12s+8
\end{align*}
Clearly, $L_1-R>0$ when $s\ge n$.

\nt When $3s-2n\le 0$, i.e., $s-n\le \frac{-s}{2}$, then
\begin{align*}
L_1-R& = 8ns(3s-2n)+(14s+48n+17)(s-n)+4ns+12s+8\\
& \le (14s+48n+17)\left(\frac{-s}{2}\right)+4ns+12s+8=\frac{-14s^2-40ns+7s+16}{2}<0.
\end{align*}

\nt
When $3s-2n=1$, i.e., $n=\frac{3s-1}{2}$. $L_1-R=\frac{-50s^2+105s+9}{2}$.

\nt When $3s-2n=2$, i.e., $n=\frac{3s-2}{2}$. $L_1-R=\frac{-26s^2+187s-46}{2}$.

\nt  When $3s-2n=3$, i.e., $n=\frac{3s-3}{2}$. $L_1-R=\frac{-2s^2+253s-149}{2}$.

\nt The above three differences are not zero because their determinants are not prefect square.

\nt When $3s-2n\ge 4$ and $n>s$, then $s-n\ge \frac{-s+4}{2}$. Thus
\begin{align*}
L_1-R& = 8ns(3s-2n)+(14s+48n+17)(s-n)+4ns+12s+8\\
& \ge 32ns +(14s+48n+17)\left(\frac{-s+4}{2}\right)+4ns+12s+8\\
& = \frac{-14s^2+63s+24ns+192n+84}{2} >0. \tag{since $n>s$.}
\end{align*}
Since $L_2>L_1$, $L_2-R>0$ when $3s-2n\ge 4$. So we only need to consider $3s-2n\le 3$, i.e., $n\ge \frac{3s-3}{2}$.

\nt When $3s-2n\le 0$, then $s-n\le \frac{-s}{2}$. We have \begin{align*}
L_2-R & =  24ns^2+42ns+14s^2+31s-16sn^2-48n^2-11n+13\\
& = 8ns(3s-2n)+(14s+48n+17)(s-n)+8ns+14s+6n+13\\
& \le (14s+48n+17)\left(\frac{-s}{2}\right)+8ns+14s+6n+13\\
& = \frac{-14s^2+11s-32ns+ 12n+26}{2}<0.
\end{align*}
When $3s-2n=3$, i.e., $n=\frac{3s-3}{2}$. $L_2-R=\frac{10s^2+263s-157}{2}> 0$.

\nt When $3s-2n=2$, i.e., $n=\frac{3s-2}{2}$. $L_2-R=\frac{-14s^2+201s-48}{2}$.

\nt When $3s-2n=1$, i.e., $n=\frac{3s-1}{2}$. $L_2-R=\frac{-38s^2+123s+13}{2}$.

\nt The above two differences are not zero because their determinants are not prefect square.

\nt Thus, condition (ii) holds.

\nt By Theorem~\ref{thm-GVOm-1} and $\chi(C_{4n+2}\vee O_{2s+1})=3$, we have $\chi_{ltna}(C_{4n+2}\vee O_{2s+1}) = 3$.
\end{proof}

\nt As a natural extension of Theorem~\ref{thm-GVOm-1}, we now have the following theorem.

\begin{theorem}\label{thm-GVOdd}  Let $G$ be a $(p,q)$-graph with $V(G) = \{u_i \mid 1\le i\le p\}$ satisfying conditions (a), (b) and (A) above. Suppose $H$ is an $(m,\n)$-graph such that  $V(H)=\{v_j\;|\; 1\le j\le m\}$ and $\chi(H) = \chi_{ltna}(H) = t'\ge 2$. Let $h$ be a local total neighborhood antimagic $t'$-coloring of $H$. If for $1\le l < l' \le t$, and $1\le j \ne j' \le m$,
\begin{enumerate}[(i)]
\item $\a_l + \b_l \ne \a_{l'} + \b_{l'}$, and
\item $\a_l + \b_l + m(2p+2q+pm) + \sum\limits^m_{j=1} h(v_j)\ne \sum\limits^p_{i=1} g(u_i) + p(pm+1)/2 + p(p+q) + h^+_{tn}(v_j) + 2d_H(v_j)(p+q+pm)$,
\item $h^+_{tn}(v_j) = h^+_{tn}(v_{j'})$ implies that $d_H(v_j) = d_H(v_{j'})$,
\item $h^+_{tn}(v_j) \ne h^+_{tn}(v_{j'})$ implies that $h^+_{tn}(v_j) - h^+_{tn}(v_{j'}) \ne 2(p+q+pm)[d_H(v_{j'}) - d_H(v_j)]$,
\end{enumerate}
then $\chi_{ltna}(G \vee H) \le  t+t'$. The equality holds if $\chi(G) = t$.
\end{theorem}

\begin{proof} Let $F = G \vee H$.  Define a bijective total labeling $f : V(F) \cup E(F) \to [1,p+q+pm+m+\n]$ such that
\begin{enumerate}[(1)]
\item $f(x) = g(x)$ for $x\in V(G) \cup E(G)$,
\item $f(u_iv_j) = a_{i,j} + p + q$, where $a_{i,j}$ is the $(i,j)$-entry of $A_{p\times m}$,
\item $f(y) = h(y) + p+q+pm$ for $y\in V(H) \cup E(H)$.
\end{enumerate}

Keeping the notation used in the proof of Theorem~\ref{thm-GVOm-1}, we have
\begin{align*}
f^+_{tn}(u_{\s_l+k}) & = g^+_{tn}(u_{\s_l+k}) +  \r_{\s_l+k} + m(p+q) + \sum^m_{j=1} h(v_j)\\
         & = \a_l + \b_l + m(2p+2q+pm) +  \sum^m_{j=1} h(v_j), \quad 1\le k\le s_l.\\
f^+_{tn}(v_j) & = \sum^p_{i=1} g(u_i) + p(pm+1)/2 + p(p+q) + h^+_{tn}(v_j) + 2d_H(v_j)(p+q+pm).
\end{align*}
\nt Therefore, all the vertices in $V_l$ have the same vertex color $\a_l + \b_l + m(2p+2q+pm) + \sum^m_{j=1} h(v_j)$. By (i), we have $f^+_{tn}(u_{\s_l+k}) \ne f^+_{tn}(u_{\s_{l'}+k'})$ for all $1\le l< l'\le t$,  $1\le k\le s_l$, $1\le k' \le  s_{l'}$. By (ii), we have $f^+_{tn}(u_{\s_l+k})\ne f^+_{tn}(v_j)$ for $1\le l\le t$ and $1\le j\le m$. By (iii) and (iv), we have $f^+_{tn}(v_j) = f^+_{tn}(v_{j'})$ if and only if $h^+_{tn}(v_j) = h^+_{tn}(v_{j'})$ for $1\le j\ne j' \le m$. Thus, $f$ is a local total neighborhood antimagic $(t+t')$-coloring of $F$ so that $\chi_{ltna}(F) \le t+ t'$. Since $\chi_{ltna}(F) \ge \chi(F) = \chi(G) + \chi(H) = \chi(G) + t'$, the equality holds if $\chi(G) = t$.
\end{proof}

\begin{corollary}\label{cor-GVOddreg} Suppose $G$, $H$, $g$ and $h$ are defined in Theorem~~\ref{thm-GVOdd}. Suppose $H$ is $d$-regular.
If for $1\le l\le l' \le t$ and $1\le j\ne j' \le m$,
\begin{enumerate}[(i)]
\item $\a_l + \b_l \ne \a_{l'} + \b_{l'}$, and
\item $\a_l + \b_l + m(2p+2q+pm) + \sum\limits^m_{k=1} h(v_k)  \ne \sum\limits^p_{i=1} g(u_i) + p(pm+1)/2 + p(p+q) + h^+_{tn}(v_j) + 2d(p+q+pm)$,
\end{enumerate}
then $\chi_{ltna}(G\vee H) \le t+t'$.
The equality holds if $\chi(G) = t$.
\end{corollary}

\begin{theorem}\label{thm-rK2VW2s}  For $r\ge 1$ and $s\ge 2$, $\chi_{ltna}(rK_2\vee W_{2s}) = 5$. \end{theorem}

\begin{proof} We keep the notation and the labeling $g$  in proving Theorem~\ref{thm-rK2VOm} except that $O_{2s+1}$ is now replaced by $H = W_{2s}$ so that $V(H) = \{v_j \mid 1\le j\le 2s+1\}$ and $E(H) = \{v_1v_{2s}, v_jv_{j+1} \mid 1\le j\le 2s-1\}\cup \{v_jv_{2s+1} \mid 1\le j\le 2s\}$.

\ms\nt Suppose $s=2n+1$. In~\cite[Theorem 4.19]{GLSY}, the authors show that $W_{2s}$ admits a local total neighborhood antimagic 3-coloring $h$ such that sum of all the vertex labels under $h$ is $\sum^{2s+1}_{j=1} h(v_j) = 1+2+\cdots+2s + (12n+7) = s(2s+1) + 6s+1 =  2s^2+7s+1$.  Moreover, the 3 induced vertex colors under $h$ are given by $h^+_{tn}(v_1) = (41s+7)/2$, $h^+_{tn}(v_2)  = (35s+7)/2$ and $h^+_{tn}(v_{2s+1}) = 2s(6s+1)$.

\ms\nt We now let $p=2r$, $q=r$ and $m=2s+1$. Similar to the proof of Theorem~\ref{thm-rK2VOm}, we also have $\a_1=3r, \a_2=4r, \b_1=s(4rs+3r+1) + 1, \b_2 = s(4rs+5r+1)+2$. Therefore, $\a_1+\b_1 \ne \a_2 + \b_2$ and condition (i) of Theorem~\ref{thm-GVOdd} holds. Clearly, $h$ satisfies condition~(iii).

\ms\nt Now $h^+_{tn}(v_1)-h^+_{tn}(v_{2s+1})= \frac{41s+7}{2}-2s(6s+1)=-\frac{24s^2-37s-7}{2}<0$ and $2(p+q+pm)[d_H(v_{2s+1}) - d_H(v_1)]>0$.

\nt Since $d_H(v_1)=d_H(v_2)$ and $h^+_{tn}(v_2)-h^+_{tn}(v_{2s+1}) <h^+_{tn}(v_1)-h^+_{tn}(v_{2s+1})$, condition~(iv) holds.

\nt We shall show that condition (ii) holds. We have $L_1, L_2, R_1, R_2, R_3$ as follows.
\begin{align*}
L_1  &=  \a_1 + \b_1 + m(2p+2q+pm) + \sum^m_{j=1} h(v_j) \\
&=   3r + s(4rs+3r+1) +1 + (2s+1)(6r + 2r(2s+1)) +  2s^2+7s+1         \\
 & = r(12s^2+23s+11)+2s^2+8s+2, \\
L_2 & =  \a_2+\b_2 + m(2p+2q+pm) +\sum^m_{j=1} h(v_j) \\
&=  4r +  s(4rs+5r+1)+2 +   (2s+1)(6r + 2r(2s+1)) +  2s^2+7s+1         \\
 & = r(12s^2 + 25s+12) + 2s^2+8s+3,\\
R_1 & =  \sum\limits^p_{i=1} g(u_i) + p(pm+1)/2 + p(p+q) + h^+_{tn}(v_1) + 2d_H(v_1)(p+q+pm) \\
&=  r(4r+1) + r(4rs+2r+1) + 2r(3r) + (41s+7)/2 + 2(3)(3r + 2r(2s+1))    \\
 & =   r^2(4s+12) + r(24s+32) + (41s+7)/2 \\
R_2 & =  \sum\limits^p_{i=1} g(u_i) + p(pm+1)/2 + p(p+q) + h^+_{tn}(v_2) + 2d_H(v_2)(p+q+pm) \\
&=  r(4r+1) + r(4rs+2r+1) + 2r(3r) + (35s+7)/2 + 2(3)(3r + 2r(2s+1))        \\
  & =    r^2(4s+12) + r(24s+32) +  (35s+7)/2\\
R_3 & =  \sum\limits^p_{i=1} g(u_i) + p(pm+1)/2 + p(p+q) + h^+_{tn}(v_{2s+1}) + 2d_H(v_{2s+1})(p+q+pm) \\
&= r(4r+1) + r(4rs+2r+1) + 2r(3r) + 2s(6s+1) + 2(2s)(3r + 2r(2s+1))\\
  & =  r^2(4s+12) + r(16s^2+20s+2) + 12s^2 + 2s
\end{align*}

\nt $R_3-L_2=4r^2s+4rs^2+12r^2+10s^2-5rs-6s-3$ is clearly positive. Thus $R_3>L_2>L_1$.

\ms\nt Consider $2(R_2-L_2)=8r^2s+24r^2-24rs^2-4s^2-2rs+40r+19s+1=\psi(r)$. Since $\psi'(r)=(16s+48)r-(24s^2+2s-40)>0$ when $r\ge 3s-5$, $\psi(r)$ is increasing. Thus
$2(R_2-L_2)\ge\psi(3s-5)=86s^2-371s+401>0$ when $r\ge 3s-5$. Then $R_1>R_2>L_2>L_1$.

\ms\nt Let us consider $L_1-R_1$.
\begin{align*}2(L_1-R_1)& =24rs^2-8r^2s+4s^2-2rs-24r^2-42r-25s-3\\&=2rs(4(3s-r)-1)+4s^2-24r^2-42r-25s-3.\end{align*}
If $3s-r\ge 10$, i.e., $s\ge \frac{r+10}{3}$, then $2(L_1-R_1)\ge 78rs+4s^2-24r^2-42r-25s-3=\phi(s)$. Clearly, $\phi(s)$ is increasing when $s>0$,
\begin{align*}2(L_1-R_1)& \ge 78r\left(\frac{r+10}{3}\right)+4\left(\frac{r+10}{3}\right)^2-24r^2-42r-25\left(\frac{r+10}{3}\right)-3\\
&=\frac{1}{9}(22r^2+1967r-377)>0.\end{align*}
Hence $L_2>L_1>R_1>R_2$.

\ms\nt We only need to consider $3s-r\in[6,9]$. Now\\
$2(L_2-R_1)=24rs^2-8r^2s+4s^2+2rs-24r^2-40r-25s-1$,\\
$2(R_2-L_1)=8r^2s-24rs^2+24r^2+2rs-4s^2+42r+19s+3$.

\begin{enumerate}[1.]
\item $r=3s-6$:

$2(L_1-R_1)=-74s^2+437s-615\ne 0$.

$2(L_2-R_1)=-62s^2+419s-625\ne 0$.

$2(R_2-L_2)=62s^2-425s+625\ne 0$.

$2(R_2-L_1)=74s^2-443s+615\ne 0$.

\item $r=3s-7$:

$2(L_1-R_1)=-50s^2+479s-885=2(2s-5)(25s-177)$. So $L_1\ne R_1$.

$2(L_2-R_1)=-38s^2+457s-897\ne 0$.

$2(R_2-L_2)=38s^2-463s+897\ne 0$.

$2(R_2-L_1)=50s^2-485s+885\ne 0$.

\item $r=3s-8$:

$2(L_1-R_1)=-26s^2+505s-1203\ne 0$.

$2(L_2-R_1)=-14s^2+479s-1217\ne 0$.

$2(R_2-L_2)=14s^2-485s+1217\ne 0$.

$2(R_2-L_1)=26s^2-511s+1203\ne 0$.

\item $r=3s-9$:

$2(L_1-R_1)=-2s^2+515s-1569\ne 0$.

$2(L_2-R_1)=10s^2+485s-1585>0$ since $s\ge 5$, hence $L_2>R_1>R_2$.

$2(R_2-L_1)=2s^2-521s+1569\ne 0$.
\end{enumerate}
Note that, all above discriminants except one are not prefect squares. Thus condition (ii) holds.

\ms\nt We now consider $s = 2n$. In~\cite[Theorem 4.19]{GLSY}, the authors also show that $W_{2s}$ admits a local total neighborhood antimagic 3-coloring $h$ such that

\nt{\bf Case (1)} when $n = 1$, the sum of all the vertex labels under $h$ is $\sum^5_{j=1} h(v_j) = 21$, and the 3 induced vertex colors are given by $h^+_{tn}(v_1) = 41$, $h^+_{tn}(v_2) = 39$ and $h^+_{tn}(v_5) = 54$;

\nt{\bf Case (2)} when $n\ge 2$, the sum of all the vertex labels under $h$ is $\sum^{2s+1}_{j=1} h(v_j) = 1+2+\cdots + 4n + (7n+1) = 2n(4n+1)+7n+1 = 8n^2 + 9n + 1 = (n+1)(8n+1) = (s+2)(4s+1)/2$, and the 3 induced vertex colors are given by $h^+_{tn}(v_1) = 35n+5 = (35s+10)/2$, $h^+_{tn}(v_2) = 31n+5 = (31s+10)/2$ and $h^+_{tn}(v_{2s+1}) = 8n(6n+1) = 4s(3s+1)$.

\ms\nt We also let $p=2r, q=r, m = 2s+1, \a_1 = 3r, \a_2 = 4r, \b_1 = s(4rs+3r+1) + 1$ and $\b_2 = s(4rs+5r+1)+2$. Similar to the argument for $s = 2n+1$, It is easy to verify that conditions (i), (iii) and (iv) of Theorem~\ref{thm-GVOdd} hold.

\ms\nt We shall show condition (ii) holds. The formula of $L_1, L_2, R_1, R_2, R_3$ are the same as when $s=2n+1$.

\nt{\bf Case (1)} Since $s=2, m=5$, we have $L_1, L_2, R_1, R_2, R_3$ as follows.
\begin{align*}
L_1  & = 3r + 2(8r + 3r+1) + 1 + 5(6r + 10r) + 21 = 105r + 24, \\
L_2 &=4r + 2(8r + 5r + 1) + 2 + 5(6r + 10r) + 21 = 110r + 25,\\
R_1 & = r(4r+1) + r(10r+1) + 2r(3r) + 41 + 2(3)(2r + r + 10r) = 20r^2 + 80r + 41, \\
R_2 & = r(4r+1) + r(10r+1) + 2r(3r) + 39 + 2(3)(2r + r + 10r)= 20r^2 + 80r + 39, \\
R_3 & = r(4r+1) + r(10r+1) + 2r(3r) + 54 + 2(4)(2r+r+10r)= 20r^2 + 106r + 54.
\end{align*}
Now,
$R_2 - L_2  = 20r^2 - 30r + 14 > 0$ for $r\ge 1$, we have $R_3 > R_1 > R_2 > L_2 > L_1$. Therefore, conditions (ii) holds.

\ms\nt{\bf Case (2)} We have $L_1, L_2, R_1, R_2, R_3$ as follows.
\begin{align*}
L_1  & = 3r + s(4rs+3r+1) +1 + (2s+1)(6r + 2r(2s+1)) + (s+2)(4s+1)/2\\
 & = r(12s^2+23s+11) + 2s^2 + 11s/2 + 2, \\
L_2 & = 4r + s(4rs+5r+1) +2 + (2s+1)(6r + 2r(2s+1)) + (s+2)(4s+1)/2\\
 & = r(12s^2+25s+12)+2s^2+11s/2+3,\\
R_1  & = r(4r+1) + r(4rs+2r+1) + 2r(3r) + (35s+10)/2 + 2(3)(3r + 2r(2s+1))\\
 & =  r^2(4s+12) + r(24s+32) + (35s+10)/2, \\
R_2  & = r(4r+1) + r(4rs+2r+1) + 2r(3r) + (31s+10)/2 + 2(3)(3r + 2r(2s+1))\\
 & =  r^2(4s+12) + r(24s+32) + (31s+10)/2, \\
R_3 & =   r(4r+1) + r(4rs+2r+1) + 2r(3r) + 4s(3s+1) + 2(2s)(3r + 2r(2s+1)) \\
 & =   r^2(4s+12) + r(16s^2+20s+2) + 12s^2 + 4s.
\end{align*}
\nt $2(R_3-L_2)=8r^2s+20rs^2+24r^2+15rs+20s^2-8r-2s-4$ is clearly positive. Thus $R_3>L_2>L_1$.

\nt $2(R_2-L_2)=8r^2s+24r^2-24rs^2-4s^2-2rs+40r+20s+4$. Similar to the case when $s=2n+1$, we have $R_1>R_2>L_2>L_1$ when $r\ge 3s-5$.

\nt $2(L_1-R_1)=24rs^2-8r^2s+4s^2-2rs-24r^2-42r-24s-6$. Similar to the case when $s=2n+1$, we have
$2(L_1-R_1)\ge \frac{1}{9}(22r^2+1970r-374)>0$ when $3s-r\ge 10$. Hence $L_2>L_1>R_1>R_2$.

\ms\nt We only need to consider $3s-r\in[6,9]$. Now\\
$2(L_2-R_1)=24rs^2-8r^2s+4s^2+2rs-24r^2-40r-24s-4$,\\
$2(R_2-L_1)=8r^2s-24rs^2+24r^2+2rs-4s^2+42r+20s+6$.

\begin{enumerate}[1.]
\item $r=3s-6$:

$2(L_1-R_1)=-74s^2+438s-618\ne 0$.

$2(L_2-R_1)=-62s^2+420s-628\ne 0$.

$2(R_2-L_2)=62s^2-424s+628\ne 0$.

$2(R_2-L_1)=74s^2-442s+618\ne 0$.

\item $r=3s-7$:

$2(L_1-R_1)=-50s^2+480s-888\ne 0$.

$2(L_2-R_1)=-38s^2+458s-900\ne 0$.

$2(R_2-L_2)=38s^2-462s+900\ne 0$.

$2(R_2-L_1)=50s^2-484s+888\ne 0$.

\item $r=3s-8$:

$2(L_1-R_1)=-26s^2+506s-1206\ne 0$.

$2(L_2-R_1)=-14s^2+480s-1220\ne 0$.

$2(R_2-L_2)=14s^2-484s+1220\ne 0$.

$2(R_2-L_1)=26s^2-510s+1206\ne 0$.

\item $r=3s-9$:

$2(L_1-R_1)=-2s^2+516s-1572\ne 0$.

$2(L_2-R_1)=10s^2+486s-1588>0$ since $s\ge 4$, hence $L_2>R_1>R_2$.

$2(R_2-L_1)=2s^2-520s+1572\ne 0$.
\end{enumerate}
Note that, all above discriminants are not prefect squares.
\end{proof}

\begin{theorem}\label{thm-rK2VCmodd}  For $r, s\ge 1$, $\chi_{ltna}(rK_2 \vee C_{2s+1}) = 5$.   \end{theorem}

\begin{proof} We keep the notation in proving Theorem~\ref{thm-rK2VOm} except that $O_{m}$ is now replaced by $H = C_m$, $m = 2s+1\ge 3$. From the proof of Theorem 4.5  in~\cite{GLSY}, we know that $H$ admits a local total neighborhood antimagic 3-coloring $h$ given by $h^+_{tn}(v_1) = 5s+6$, $h^+_{tn}(v_2) = 7s+7$ and $h^+_{tn}(v_3) = 9s+8$. Moreover, the sum of all the vertex labels under $h$ is $\sum\limits^{3s+2}_{i=s+2} i = (2s+1)(2s+2)$. Since $H$ is 2-regular, we shall show that the two conditions of Corollary~\ref{cor-GVOddreg} holds.  Now, we have $p = 2r$, $q = r$. Moreover, $\a_1 =3r$, $\a_2 = 4r$, $\b_1 = s(4rs+3r+1)+1$ and $\b_2 = s(4rs+5r+1)+2$.

\ms\nt Similar to the proof of Theorem~\ref{thm-rK2VOm}, we have $\a_1 + \b_1 \ne \a_2 + \b_2$. Thus, condition (i) holds.  We next consider condition (ii). We now have $L_1, L_2, R_1, R_2, R_3$ as follows.
\begin{align*}
L_1  &=  \a_1 + \b_1 + m(2p+2q+pm) + \sum^m_{j=1} h(v_j) \\
 & =  3r + s(4rs+3r+1) +1 + (2s+1)(6r + 2r(2s+1))  + (2s+1)(2s+2)          \\
 & = r(12s^2 + 23s+11) + 4s^2 + 7s + 3, \\
L_2 & =  \a_2+\b_2 + m(2p+2q+pm) +\sum^m_{j=1} h(v_j) \\
 & =   4r + s(4rs+5r+1) +2 + (2s+1)(6r + 2r(2s+1))  + (2s+1)(2s+2)                  \\
 & =  r(12s^2 + 25s+12) + 4s^2 + 7s + 4, \\
R_1 & =  \sum\limits^p_{i=1} g(u_i) + p(pm+1)/2 + p(p+q) + h^+_{tn}(v_1) + 2d(p+q+pm) \\
 & =     r(4r+1) + r(4rs+2r+1) + 2r(3r) + (5s+6) + 4(3r + 2r(2s+1))        \\
 & = r^2(4s+12) + r(16s+22) + 5s+6, \\
R_2 & =  \sum\limits^p_{i=1} g(u_i) + p(pm+1)/2 + p(p+q) + h^+_{tn}(v_2) + 2d(p+q+pm) \\
 & =    r(4r+1) + r(4rs+2r+1) + 2r(3r) + (7s+7) + 4(3r + 2r(2s+1))                 \\
 & =  r^2(4s+12) + r(16s+22) + 7s+7, \end{align*}
 \begin{align*}
R_3 & =  \sum\limits^p_{i=1} g(u_i) + p(pm+1)/2 + p(p+q) + h^+_{tn}(v_3) + 2d(p+q+pm) \\
 & =  r(4r+1) + r(4rs+2r+1) + 2r(3r) + (9s+8) + 4(3r + 2r(2s+1))      \\
 & =  r^2(4s+12) + r(16s+22) + 9s+8.
\end{align*}
Now
\begin{align} R_1-L_2 & = 4r^2s-12rs^2+12r^2-4s^2 -9rs+10r-2s+2\nonumber\\
& = 4rs(r-3s+2)+(r-3s+2)(6r+s+2)+(2r-s)(2r+s-2)+2r^2-2. \label{eq-K2Vcycle-1}\\
R_3-L_1 & = 4r^2s-12rs^2-7rs+12r^2-4s^2+11r+2s+5\nonumber\\
& = 4rs(r-3s+7)+(r-3s+7)(6r+s+2)+6r(r-3s)-s^2-33r+s-9.\label{eq-K2Vcycle-2}
\end{align}
When $s=1$, $R_1-L_2=16r^2-11r-4>0$. Hence $R_3>R_2>R_1>L_2>L_1$.

\ms\nt So, following we only consider $s\ge 2$.

\nt Suppose $r\ge 3s-2$. Then $r-s\ge 2s-2\ge 2$. From \eqref{eq-K2Vcycle-1}, we have $R_1-L_2>0$. Hence $R_3>R_2>R_1>L_2>L_1$.

\nt Suppose $r\le 3s-7$. From \eqref{eq-K2Vcycle-2}, we have
\[R_3-L_1\le 6r(-7)-s^2-33r+s-9=-75r-s(s-1)-9<0.\]
 Hence $L_2>L_1>R_3>R_2>R_1$.

\ms\nt So we need to consider $r\in [3s-6, 3s-3]$.
\begin{enumerate}[1)]
\item $r=3s-3$. Then $R_1-L_2 = 41s^2-125s+80=(41s-2)(s-3)+74$.
So when $s\ge 3$, we have  $R_3>R_2>R_1>L_2>L_1$.

\nt When $s=2$. Then $r=3$. We have $L_1=348$, $L_2=364$, $R_1=358$, $R_2=363$, $R_3=368$.

\item $r=3s-4$. Then $R_1-L_2 = 29s^2-160s+154=(29s-15)(s-5)+79>0$ if and only if $s\ge 5$ (since our range of $s$ is $s\ge 2$).

$R_3-L_1 = 35s^2-161s+153=7(5s-3)(s-4)+69>0$ if and only if $s\ge 4$.

So when $s\ge 5$, we have  $R_3>R_2>R_1>L_2>L_1$. When $2\le s\le 3$, we have $L_2>L_1>R_3>R_2>R_1$.

When $s=4$. Then $r=8$. We have $L_1=2455$, $L_2=2528$, $R_1=2506$, $R_2=2515$, $R_3=2524$.

\item $r=3s-5$. Then $R_1-L_2 =17s^2-187s+252=17(s-10)(s-1)+82>0$ if and only if $s\ge 10$.

$R_3-L_1 = 23s^2-190s+250=(s-7)(23s-29)+47>0$ if and only if $s\ge 7$.

So when $s\ge 10$, we have  $R_3>R_2>R_1>L_2>L_1$. When $2\le s\le 6$, we have $L_2>L_1>R_3>R_2>R_1$.

Now, we have to deal with $7\le s\le 9$.

$R_3-L_2 =R_1+4s+2-L_2=17s^2-183s+254=(s-10)(17s-13)+124>0$ if and only if $s\ge 10$.

$R_1-L_1=R_3-4s-2-L_1=23s^2-194s+248=(s-7)(23s-33)+17>0$ if and only if $s\ge 7$.

Thus, we have $L_2>R_3>R_2>R_1>L_1$ if $7\le s\le 9$.

\item $r=3s-6$. Then $R_1-L_2= 5s^2-206s+374=(s-40)(5s-6)+134>0$ if and only if $s\ge 40$.

Next $R_3-L_1 = 11s^2-211s+371=(s-18)(11s-13)+137>0$ if and only if $s\ge 18$.

So when $s\ge 40$, we have  $R_3>R_2>R_1>L_2>L_1$. When $2\le s\le 17$, we have $L_2>L_1>R_3>R_2>R_1$.

Now, we are looking for $18\le s\le 39$.

$R_3-L_2=5s^2-202s+376=(s-39)(5s-7)+103 >0$ if and only if $s\ge 39$.

$R_1-L_1=11s^2-215s+369=(s-18)(11s-17)+63>0$ if and only if $s\ge 18$.

Thus we have $L_2>R_3>R_2>R_1>L_1$ when $18\le s\le 38$.

When $s=39$. Then $r=111$. Then $L_1=2133120$, $L_2=2141890$, $R_1=2141835$, $R_2=2141914$, $R_3=2141993$.
\end{enumerate}
Combining all cases, we have the theorem.
\end{proof}

\begin{theorem}\label{thm-C4n+2VW2s}  For $n\ge 1, s\ge 2$, $\chi_{ltna}(C_{4n+2} \vee W_{2s}) = 5$.   \end{theorem}

\begin{proof} We keep the notation in proving Theorem~\ref{thm-C4n+2VO2s+1} except that $O_{2s+1}$ is now replaced by $H = W_{2s}$ so that $V(H) = \{v_j \mid 1\le j\le 2s+1\}$ and $E(H) = \{v_1v_{2s}, v_jv_{j+1} \mid 1\le j\le 2s-1\} \cup\{v_jv_{2s+1}\mid 1\le j\le 2s\}$. Recall that $C_{4n+2}$ admits a bijective total labeling with $g$ such that $$\{g^+_{tn}(u_{4n}), g^+_{tn}(u_{4n-2}),\dots, g^+_{tn}(u_{4}), g^+_{tn}(u_2), g^+_{tn}(u_{4n+2})\}$$ forms an arithmetic sequence with first term $\a_1=15n+9$ and $d_1=-1$ when $n\ge 2$. Moreover, $$\{g^+_{tn}(u_{5}), g^+_{tn}(u_{7}),\dots, g^+_{tn}(u_{4n-1}), g^+_{tn}(u_{4n+1}), g^+_{tn}(u_{1}), g^+_{tn}(u_3)\}$$ forms an arithmetic sequence with first term $\a_2=19n+13$ and $d_2=-1$ when $n\ge 2$. When $n=1$, we get $\{g^+_{tn}(u_{4}), g^+_{tn}(u_{2}), g^+_{tn}(u_{6})\}=\{24,23,22\}$ with $\a_1=24$ and $d_1=-1$ (respectively, $\{g^+_{tn}(u_{5}), g^+_{tn}(u_{1}), g^+_{tn}(u_{3})\}=\{30, 29, 28\}$) with $\a_2=30$ and $d_2=-1$. Moreover, the sum of all the vertex labels under $g$ is $1+2+\cdots + (4n+2) = (2n+1)(4n+3)$.

\ms\nt Suppose $s\ge 3$ is odd. In~\cite[Theorem 4.19]{GLSY}, the authors show that $W_{2s}$ admits a local total neighborhood antimagic 3-coloring $h$ given by $h^+_{tn}(v_1) = (41s+7)/2$, $h^+_{tn}(v_2) = (35s+7)/2$ and $h^+_{tn}(v_{2s+1}) = 2s(6s+1)$. Moreover, the sum of all the vertex labels under $h$ is $2s^2+7s+1$. We now let $p = 2r = 4n+2=q$ and $m=2s+1$. Similar to the proof of Theorem~\ref{thm-C4n+2VO2s+1},  we have $\b_1 = s(8ns+4s+6n+4)+1$, $\b_2 = s(8ns+4s+10n+6)+r+1$ and $-d_1=-d_2 = 1$.

\ms\nt Similar to the proof of Theorem~\ref{thm-rK2VW2s}, we know conditions (i), (iii), (iv) of Theorem~\ref{thm-GVOdd} hold.  We next consider condition (ii). For $n=1$, we have $L_1, L_2, R_1, R_2, R_3$ as follows.
\begin{align*}
L_1  &=  \a_1 + \b_1 + m(2p+2q+pm) + \sum^m_{j=1} h(v_j) \\
 & =  12s^2 + 10s + 25 + (2s+1)(24 + 6(2s+1))  + 2s^2+7s+1  = 38s^2+89s+56, \\
L_2 & =  \a_2+\b_2 + m(2p+2q+pm) +\sum^m_{j=1} h(v_j) \\
 & =   12s^2+16s+34 + (2s+1)(24 + 6(2s+1))  + 2s^2+7s+1 =  38s^2 + 95s + 65, \\
R_1 & =  \sum\limits^p_{i=1} g(u_i) + p(pm+1)/2 + p(p+q) + h^+_{tn}(v_1) + 2d_H(v_1)(p+q+pm) \\
 & =     21 + 3(6(2s+1)+1) + 72 + (41s+7)/2 + 6(12s+18)  = (257s+451)/2, \\
R_2 & =  \sum\limits^p_{i=1} g(u_i) + p(pm+1)/2 + p(p+q) + h^+_{tn}(v_2) + 2d_H(v_2)(p+q+pm) \\
 & =  21 + 3(6(2s+1)+1) + 72 + (35s+7)/2 + 6(12s+18)  =  (251s+451)/2, \\
R_3 & =  \sum\limits^p_{i=1} g(u_i) + p(pm+1)/2 + p(p+q) + h^+_{tn}(v_3) + 2d_H(v_{2s+1})(p+q+pm) \\
 & =  21 + 3(6(2s+1)+1) + 72 + 2s(6s+1) + 2(2s)(12s+18) = 60s^2 + 110s+114.
\end{align*}
Since $s\ge 3$, $2(L_1-R_1)  = 76s^2-79s-339=(s-3)(76s+149)+108>0$. So $R_3 > L_2 > L_1 > R_1 > R_2$. Thus, condition (ii) holds.

\ms\nt For $n\ge 2$, we have $L_1, L_2, R_1, R_2, R_3$ as follows.
\begin{align*}
L_1  &= \a_1 + \b_1 + m(2p+2q+pm) + \sum^m_{j=1} h(v_j) \\
 & = 15n+10 + s(8ns+4s+6n+4)  + (2s+1)(2(8n+4) + (4n+2)(2s+1)) +  2s^2+7s+1          \\
 & = n(24s^2+54s+35) + 14s^2+35s+21, \\
L_2 & = \a_2+\b_2 + m(2p+2q+pm) +\sum^m_{j=1} h(v_j) \\
 & =  21n+15 + s(8ns+4s+10n+6) + (2s+1)(2(8n+4) + (4n+2)(2s+1))  + 2s^2+7s+1                  \\
 & =  n(24s^2+58s+41) + 14s^2+37s+26, \\
R_1 & =  \sum\limits^p_{i=1} g(u_i) + p(pm+1)/2 + p(p+q) + h^+_{tn}(v_1) + 2d_H(v_1)(p+q+pm) \\
 & =  (2n+1)(4n+3) + (2n+1)((4n+2)(2s+1)+1) + (4n+2)(8n+4) + (41s+7)/2 +  \\ &\quad\  2(3)(8n+4+(4n+2)(2s+1))       \\
 & =  n^2(16s+48) + n(64s+124) + (97s+107)/2, \end{align*}
 \begin{align*}
R_2 & =  \sum\limits^p_{i=1} g(u_i) + p(pm+1)/2 + p(p+q) + h^+_{tn}(v_2) + 2d_H(v_2)(p+q+pm) \\
 & =   (2n+1)(4n+3) + (2n+1)((4n+2)(2s+1)+1) + (4n+2)(8n+4) + (35s+7)/2 + \\ & \quad\  2(3)(8n+4+(4n+2)(2s+1))       \\
 & =  n^2(16s+48) + n(64s+124) + (91s+107)/2, \\
R_3 & =  \sum\limits^p_{i=1} g(u_i) + p(pm+1)/2 + p(p+q) + h^+_{tn}(v_3) + 2d_H(v_{2s+1})(p+q+pm) \\
 & =    (2n+1)(4n+3) + (2n+1)((4n+2)(2s+1)+1) + (4n+2)(8n+4) + 2s(6s+1) + \\ & \quad\  2(2s)(8n+4+(4n+2)(2s+1))      \\
 & = n^2(16s+48)+n(32s^2+64s+52)+28s^2+30s+14.
\end{align*}
Clearly, $R_3>R_1>R_2$ and $R_3>L_2>L_1$. We only need to compare the values of $R_1, R_2, L_1, L_2$.
\begin{align*}
2(L_1-R_1) & = -32sn^2+48ns^2+28s^2-96n^2-20ns-27s-178n-65\\
& = (16ns+9s+49n)(3s-2n-10)+s^2+2n^2+63s+11ns+312n-65,\\
2(R_2-L_2) & = 32sn^2-48ns^2+96n^2+12ns-28s^2+166n+17s+55\\
& =(16ns+48n+10s)(2n-3s+5)+2s^2-33s+56ns-74n+55\\
& = (16ns+48n+10s)(2n-3s+5)+2s^2+(19n-33)s+(37s-74)n+55.
\end{align*}
Thus, if $2n\le 3s-10$, then $L_2>L_1>R_1>R_2$; if $2n\ge 3s-5$, then $R_1>R_2>L_2>L_1$. Now, since $s$ is odd, we need to deal with $2n\in\{3s-9, 3s-7\}$. Recall that $s\ge 3$.
Since $R_1=R_2+3s$,
\begin{align*}
2(L_1-R_2) & = -32sn^2+48ns^2+28s^2-96n^2-20ns-21s-178n-65\\
2(R_1-L_2) & = 32sn^2-48ns^2+96n^2+12ns-28s^2+166n+23s+55.
\end{align*}

\begin{enumerate}[1.]
\item Suppose $2n=3s-9$. Now, $2(L_1-R_1)=-2s^2+444s-1208=-2(s-219)(s-3)+106$. So $L_1-R_1>0$ if and only if $3\le s\le 219$.

Thus $L_2>L_1>R_1>R_2$ when $3\le s\le 219$.

$2(R_1-L_2)=-10s^2-430s+1252=-2(s-3)(5s+230)-128<0$. So $L_2>R_1>R_2$.

$2(L_1-R_2)=-2s^2+450s-1208=-2(s-222)(s-3)+124$. So $L_1>R_2$ if and only if $3\le s\le 222$.

Thus, $L_2>R_1>L_1>R_2$ when $s=221$; and $L_2>R_1>R_2>L_1$ when $s\ge 223$.

\item Suppose $2n=3s-7$. Now, $2(L_1-R_1)=-50s^2+392s-618=-2(s-5)(25s-71)+92>0$ if and only if $3\le s\le 5$. Thus $L_2>L_1>R_1>R_2$ when $3\le s\le 5$

$2(R_2-L_2)=38s^2-392s+650=2(s-8)(19s-44)-54<0$ if and only if $3\le s\le 8$. Thus $R_1>R_2>L_2>L_1$ when $s\ge 9$.

As $s$ is odd, we are left with $s=7=n$. Now, $2(R_1 - L_2) = -190$ and $2(L_1 - R_2) = -282$. Thus,  $L_2>R_1>R_2>L_1$ when $s=7$.
\end{enumerate}

\ms\nt We now consider $s\ge2$ is even.  In~\cite[Theorem 4.19]{GLSY}, the authors show that $W_{2s}$ admits a local total neighborhood antimagic 3-coloring $h$ such that

\nt{\bf Case (1)} when $s=2$, the sum of all the vertex labels under $h$ is 21 and the 3 induced vertex colors are $41$, $39$, $54$;

\nt{\bf Case (2)} when $s\ge 4$, the sum of all the vertex labels under $h$ is $(s+2)(4s+1)/2$ and the 3 induced vertex colors are $(35s+10)/2$, $(31s+10)/2$, $4s(3s+1)$.

\ms\nt Recall that $p=2r=4n+2=q$ and $m = 2s+1$. Similar to the proof of Theorem~\ref{thm-rK2VW2s}, we know conditions (i), (iii), (iv) of Theorem~\ref{thm-GVOdd} hold.  We next consider condition (ii).

\ms\nt Suppose $n=1$, $s=2$, we have $L_1, L_2, R_1, R_2, R_3$ as follows.
\begin{enumerate}[(i)]
\item $L_1  = 12s^2 + 10s + 25 + (2s+1)(24 + 6(2s+1))  +21 = 384$,
\item $L_2 = 12s^2+16s+34 + (2s+1)(24 + 6(2s+1)) + 21 = 405$,
\item $R_1 =  21 + 3(6(2s+1)+1) + 72 + 41 + 6(12s+18) = 479$,
\item $R_2 = 21 + 3(6(2s+1)+1) + 72 + 39 + 6(12s+18) = 477$,
\item $R_3 = 21 + 3(6(2s+1)+1) + 72 + 54 + 8(12s+18) = 576$.
\end{enumerate}
Thus, condition (ii) is satisfied.

\ms\nt Suppose $n=1, s\ge 4$, we have $L_1, L_2, R_1, R_2, R_3$ as follows.
\begin{eqnarray*}
L_1  &=& 12s^2 + 10s + 25 + (2s+1)(24 + 6(2s+1))  + (s+2)(4s+1)/2\\
 &=&  (76s^2 + 173s+112)/2, \\
L_2 &=& 12s^2+16s+34 + (2s+1)(24 + 6(2s+1))  + (s+2)(4s+1)/2\\
 &=&  (76s^2 + 185s+130)/2, \\
R_1 &=& 21 + 3(6(2s+1)+1) + 72 + (35s+10)/2 + 2(3)(12s+18) \\
 &=& (251s + 454)/2,  \\
R_2 &=& 21 + 3(6(2s+1)+1) + 72 + (31s+10)/2 + 2(3)(12s+18) \\
 &=& (247s+454)/2,  \\
R_3 &=& 21 + 3(6(2s+1)+1) + 72 + 4s(3s+1) + 2(2s)(12s+18) \\
 &=& 60s^2 + 112s + 114.
\end{eqnarray*}
Clearly, $R_3 > L_2 > L_1 > R_1 > R_2$. Thus, condition (ii) is satisfied.

\ms\nt Suppose $n\ge2, s=2$, we have $L_1, L_2, R_1, R_2, R_3$ as follows.
\begin{eqnarray*}
L_1  &=&  15n+10 + s(8ns+4s+6n+4) + (2s+1)(2(8n+4)+(4n+2)(2s+1)) + 21\\
 &=& 239n+145,\\
L_2 &=& 21n+15 + s(8ns+4s+10n+6) + (2s+1)(2(8n+4)+(4n+2)(2s+1)) + 21\\
 &=& 253n+154,\\
R_1 &=& (2n+1)(4n+3) + (2n+1)((4n+2)(2s+1)+1)+ (4n+2)(8n+4) + 41 + \\
 && 2(3)(8n+4+(4n+2)(2s+1))\\
 &=& 80n^2+252n+147,\\
R_2 &=& (2n+1)(4n+3) + (2n+1)((4n+2)(2s+1)+1)+ (4n+2)(8n+4) +39 + \\
 && 2(3)(8n+4+(4n+2)(2s+1))\\
 &=& 80n^2+252n+145,\\
R_3 &=& (2n+1)(4n+3) + (2n+1)((4n+2)(2s+1)+1)+ (4n+2)(8n+4) + 54 + \\
 && 2(2s)(8n+4+(4n+2)(2s+1))\\
 &=& 80n^2+308n+188.
\end{eqnarray*}
Thus, condition (ii) is satisfied.

\ms\nt Suppose $n\ge 2, s\ge 4$, we have $L_1, L_2, R_1, R_2, R_3$ as follows.
\begin{eqnarray*}
L_1  &=&  15n+10 + s(8ns+4s+6n+4) + (2s+1)(2(8n+4)+(4n+2)(2s+1)) +\\
 && (s+2)(4s+1)/2\\
 &=& n(24s^2+54s+35) + 14s^2 + 65s/2 + 21,\\
L_2 &=& 21n+15 + s(8ns+4s+10n+6) + (2s+1)(2(8n+4)+(4n+2)(2s+1)) +\\
 && (s+2)(4s+1)/2\\
 &=& n(24s^2+58s+41) + 14s^2 + 69s/2 + 26,\\
R_1 &=& (2n+1)(4n+3) + (2n+1)((4n+2)(2s+1)+1)+ (4n+2)(8n+4) +\\
 && (35s+10)/2 + 2(3)(8n+4+(4n+2)(2s+1))\\
 &=& n^2(16s+48) + n(64s+124)+ 91s/2 + 55,\\
R_2 &=& (2n+1)(4n+3) + (2n+1)((4n+2)(2s+1)+1)+ (4n+2)(8n+4) +\\
 && (31s+10)/2 + 2(3)(8n+4+(4n+2)(2s+1))\\
 &=& n^2(16s+48) + n(64s+124) + 87s/2 + 55,\\
R_3 &=& (2n+1)(4n+3) + (2n+1)((4n+2)(2s+1)+1)+ (4n+2)(8n+4) +\\
 && 4s(3s+1) + 2(2s)(8n+4+(4n+2)(2s+1))\\
 &=& n^2(16s+48) + n(32s^2+64s+52) + 28s^2+32s+14.
\end{eqnarray*}
Clearly, $R_3>R_1>R_2$ and $R_3>L_2>L_1$. We only need to compare the values of $R_1, R_2, L_1, L_2$.
\begin{align*}
L_1-R_1 & = -16sn^2+24ns^2+14s^2-48n^2-10ns-13s-89n-34\\
& = 4(2ns+s+6n)(3s-2n-10)+2s^2+27s+6ns+151n-34,\\
R_2-L_2 & = 16sn^2-24ns^2+48n^2+6ns-14s^2+83n+9s+29\\
& = (8ns+24n+5s)(2n-3s+5)+s(s-4)+6s(n-2)+(22s-37)n+29.
\end{align*}
Thus, if $2n\le 3s-10$, then $L_2>L_1>R_1>R_2$; if $2n\ge 3s-5$, then $R_1>R_2>L_2>L_1$. Now, since $s$ is even, we need to deal with $2n\in\{3s-8, 3s-6\}$.

\begin{enumerate}[1.]
\item Suppose $2n=3s-8$. Thus, $2(L_1-R_1)=-26s^2+427s-892=-(s-13)(26s-89)+265$. So $L_1-R_1>0$ if and only if $3\le s\le 13$. Thus $L_2>L_1>R_1>R_2$ when $4\le s\le 12$.

Now, $2(R_2-L_2)= 14s^2-421s+930=(s-27)(14s-43)-231$. So $R_2-L_2<0$ if and only if $3\le s\le 27$. Thus $R_1>R_2>L_2>L_1$ when $s\ge 28$.

    Since $R_2+2s=R_1$, $2(R_1-L_2)=14s^2-417s+930=(s-27)(14s-39)-123$. So $R_1-L_2<0$ if and only if $3\le s\le 27$. Thus $L_2>R_1$ when $4\le s\le 26$.

Now, $2(L_1-R_2)=-26s^2+431s-892=-(s-14)(26s-67)+46$.  So $L_1-R_2>0$ if and only if $3\le s\le 14$. Thus, $R_2>L_1$ when $s\ge 16$.

    So, $L_1<R_2<R_1<L_2$ when $16\le s\le 26$.  When $s=14$, we have $L_2 = 97722 > R_1 =966440 > L_1=96635 > R_2=96612$.

\item Suppose $2n=3s-6$. Thus, $2(R_2-L_2)=62s^2-345s+424=(s-4)(62s-97)+36$. So $R_2-L_2>0$ for $s\ge 4$. Thus $R_1>R_2>L_2>L_1$.
\end{enumerate}
This completes the proof.
\end{proof}

\begin{theorem}\label{thm-C4n+2VC2s+1}  For $n, s\ge 1$, $\chi_{ltna}(C_{4n+2} \vee C_{2s+1}) = 5$.   \end{theorem}

\begin{proof}  We keep the notation in proving Theorem~\ref{thm-C4n+2VO2s+1} except that $O_{2s+1}$ is now replaced by $H = C_{2s+1}$ so that $V(H) = \{v_j \mid 1\le j\le 2s+1\}$ and $E(H) = \{v_jv_{j+1} \mid 1\le j\le 2s\} \cup\{v_1v_{2s+1}\}$. Recall that $C_{4n+2}$ admits a bijective total labeling with $g$ such that $$\{g^+_{tn}(u_{4n}), g^+_{tn}(u_{4n-2}),\dots, g^+_{tn}(u_{4}), g^+_{tn}(u_2), g^+_{tn}(u_{4n+2})\}$$ forms an arithmetic sequence with first term $\a_1=15n+9$ and $d_1=-1$ when $n\ge 2$. Moreover, $$\{g^+_{tn}(u_{5}), g^+_{tn}(u_{7}),\dots, g^+_{tn}(u_{4n-1}), g^+_{tn}(u_{4n+1}), g^+_{tn}(u_{1}), g^+_{tn}(u_3)\}$$ forms an arithmetic sequence with first term $\a_2=19n+13$ and $d_2=-1$ when $n\ge 2$. When $n=1$, we get $\{g^+_{tn}(u_{4}), g^+_{tn}(u_{2}), g^+_{tn}(u_{6})\}=\{24,23,22\}$ with $\a_1=24$ and $d_1=-1$ (respectively, $\{g^+_{tn}(u_{5}), g^+_{tn}(u_{1}), g^+_{tn}(u_{3})\}=\{30, 29, 28\}$) with $\a_2=30$ and $d_2=-1$. Moreover, the sum of all the vertex labels under $g$ is $1+2+\cdots + (4n+2) = (2n+1)(4n+3)$.

\ms\nt  In~\cite[Theorem 4.5]{GLSY}, the authors show that $C_{2s+1}$ admits a local total neighborhood antimagic 3-coloring $h$ given by $h^+_{tn}(v_1) = 5s+6$, $h^+_{tn}(v_2) = 7s+7$ and $h^+_{tn}(v_3) = 9s+8$. Moreover, the sum of all the vertex labels under $h$ is $(2s+1)(2s+2)$. We now let $p = 2r = 4n+2=q$ and $m=2s+1$. Similar to the proof of Theorem~\ref{thm-C4n+2VO2s+1},  we have $\b_1 = s(8ns+4s+6n+4)+1$, $\b_2 = s(8ns+4s+10n+6)+r+1$ and $-d_1=-d_2 = 1$.

\ms\nt Since $H$ is 2-regular, similar to the proof of Theorem~\ref{thm-rK2VCmodd}, we only need to show condition (ii) of Corollary~\ref{cor-GVOddreg} holds.  Note that, $h^+_{tn}(v_3)=h^+_{tn}(v_2)+2s+1=h^+_{tn}(v_1)+4s+2$. So we let $R_3=R_1+4s+2$ and $R_2=R_1+2s+1$, where

\centerline{$R_1=\sum\limits^p_{i=1} g(u_i) + p(pm+1)/2 + p(p+q) + h^+_{tn}(v_1) + 2d_H(v_1)(p+q+pm)$.}

\nt For $n=1$, we have $L_1, L_2, R_1, R_2, R_3$ as follows.

\begin{align*}
L_1  &=  \a_1 + \b_1 + m(2p+2q+pm) + \sum^m_{j=1} h(v_j) \\
 & =  12s^2 + 10s + 25 + (2s+1)(24 + 6(2s+1))  + (2s+1)(2s+2)  = 40s^2+88s+57, \\
L_2 & =  \a_2+\b_2 + m(2p+2q+pm) +\sum^m_{j=1} h(v_j) \\
 & =   12s^2+16s+34 + (2s+1)(24 + 6(2s+1))  + (2s+1)(2s+2) =  40s^2+94s+66, \\
R_1 & = 21 + 3(6(2s+1)+1) + 72 + (5s+6) + 4(12s+18)  =  89s+192, \\
R_2 & = 91s+193, \\
R_3 &  = 93s+194.
\end{align*}
Clearly, for $s=1$, $R_3 > R_2 > R_1 > L_2 > L_1$, while for $s\ge 2$, $L_2 > L_1 > R_3 > R_2 > R_1$.  Thus, condition (ii) holds.

\ms\nt For $n\ge 2$, we have $L_1, L_2, R_1, R_2, R_3$ as follows.
\begin{eqnarray*}
L_1  &=&  \a_1 + \b_1 + m(2p+2q+pm) + \sum^m_{j=1} h(v_j) \\
 & =& 15n+10 + s(8ns+4s+6n+4)  + (2s+1)(4n+2)(2s+5) +  (2s+1)(2s+2)          \\
 & =& n(24s^2+54s+35)+16s^2+34s+22, \\
L_2 & =&  \a_2+\b_2 + m(2p+2q+pm) +\sum^m_{j=1} h(v_j) \\
 & = & 21n+15 + s(8ns+4s+10n+6) + (2s+1)(4n+2)(2s+5)  + (2s+1)(2s+2)                  \\
 & = & n(24s^2+58s+41) + 16s^2+36s+27, \\
R_1  & = & (2n+1)(4n+3) + (2n+1)((4n+2)(2s+1)+1) + (4n+2)(8n+4) + (5s+6) +  \\ &  &  4(4n+2)(2s+3)       \\
 & = & n^2(16s+48) + n(48s+100) + 25s+44, \\
R_2 & =&  n^2(16s+48) + n(48s+100) + 27s+45  \\
R_3 & =& n^2(16s+48)+n(48s+100) +29s+46.
\end{eqnarray*}
Now,
\begin{align*}
R_1 - L_2 & = 16n^2s+48n^2-24ns^2-16s^2-10ns+59n-11s+17\\
& = (8ns+3n+6s+10)(2n-3s+3)+(3n-2s)(5n-s+1)+(3n-3s+1)(9n+s)\\& \quad +3s^2+2s+18n-13\\
R_3-L_1 & = 16n^2s+48n^2-24ns^2-16s^2-6ns+65n-5s+24\\
& = (8ns+24n+5s)(2n-3s+7)-s^2-40s-103n+24.
\end{align*}
Note that $2n\ge 3s-3$ implies $3n-2s\ge n+s-3\ge 0$ and $3n-3s+1\ge n-2\ge 0$ (since $s\ge 1$ and $n\ge 2$).

\nt Thus, when $2n\ge 3s-3$, then $R_3>R_2>R_1>L_2>L_1$; when $2n\le 3s-7$, then $L_2>L_1>R_3>R_2>R_1$. So we need to deal with
$2n\in\{3s-4, 3s-5, 3s-6\}$. Since $n\ge 2$, $s\ge 3$. Thus, following we always assume $s\ge 3$

\begin{enumerate}[1.]
\item $2n=3s-6$. Note that $s$ is even. $2(R_1-L_2)=10s^2-361s+544=(s-35)(10s-11)+159>0$ if and only if $s\ge 35$. Hence, $R_3>R_2>R_1>L_2>L_1$ when $s\ge 36$.

$2(R_3-L_1)=22s^2-355s+522=(s-14)(22s-47)-136<0$ if and only if $s\le 14$. Hence, $L_2>L_1>R_3>R_2>R_1$ when $s\le 14$.

$2(R_1-L_1)=2(R_3-L_1-4s-2)=22s^2-363s+518=(s-15)(22s-33)+23>0$ if and only if $s\ge 15$.

$2(R_3-L_2)=2(R_1-L_2+4s+2)=10s^2-353s+548=(s-33)(10s-23)-211<0$ if and only if $s\le 33$. Hence $L_2>R_3>R_2>R_1>L_1$ when $16\le s\le 32$.

\item $2n=3s-5$. Note that $s$ is odd. $2(R_1-L_2)=34s^2-315s+339>0$ if and only if $s\ge 9$. Hence, $R_3>R_2>R_1>L_2>L_1$ when $s\ge 9$.

$2(R_3-L_1)=46s^2-305s+323=(s-5)(46s-75)-52<0$ if and only if $s\le 5$. Hence, $L_2>L_1>R_3>R_2>R_1$ when $s\le 5$.

So only $s=7$ (i.e., $n=8$) is the remaining case. Here we have $L_1=13756$, $L_2=14047$, $R_1=13947$, $R_2=13962$ and $R_3=13977$.

\item $2n=3s-4$. Note that $s$ is even. $2(R_1-L_2)=58s^2-253s+182=(s-4)(58s-21)+98>0$ if only if $s\ge 4$.
 Hence $R_3>R_2>R_1>L_2>L_1$.
\end{enumerate}

This completes the proof.
\end{proof}

\section{Conclusion and Open problems}

This paper investigates the local total neighborhood antimagic chromatic number of join graphs whose component graphs have distinct parity orders. By constructing suitable total labelings and arithmetic-magic matrices, we establish exact values of $\chi_{ltna}(G\vee H)$ for several classes of graphs, including
$rK_2$ or even cycle joined with odd-order null graphs, wheels or cycles.

\ms\nt The key contribution lies in introducing structured matrix constructions that systematically control induced vertex colors, allowing us to extend known results and unify various cases under a common framework. These techniques provide a flexible approach for analyzing more complex join graphs and may be applicable to broader classes of graph labelings.

\ms\nt We end the paper with the followings.
\begin{problem} For $n\ge 1, s\ge 2$, determine $\chi_{ltna}(C_{4n} \vee O_{2s+1})$, $\chi_{ltna}(C_{4n}\vee W_{2s})$ and $\chi_{ltna}(C_{4n} \vee C_{2s+1})$. \end{problem}

\begin{problem} Find necessary and sufficient conditions for $\chi_{ltna}(G\vee H) = \chi(G) + \chi(H)$. \end{problem}

\begin{problem} Determine tight upper/lower bounds of $\chi_{ltna}(G\vee H)$.   \end{problem}

\begin{problem} Find relationship between $\chi_{la}(G)$ and $\chi_{ltna}(G)$. \end{problem}

\end{document}